\documentclass{amsart}
\usepackage{amsmath}
\usepackage{amssymb}
\usepackage{amsthm}

\theoremstyle{plain}
\newtheorem{theorem}{Theorem}[section]

\newtheorem{proposition}[theorem]{Proposition}
\newtheorem{lemma}[theorem]{Lemma}
\newtheorem{corollary}[theorem]{Corollary}
\newtheorem{notation}{Notation}[section]

\theoremstyle{definition}

\theoremstyle{remark}

\numberwithin{equation}{section}

\DeclareMathOperator{\s}{s}
\DeclareMathOperator{\card}{card}
\DeclareMathOperator{\ord}{Ord}
\DeclareMathOperator{\dom}{dom}

\begin{document}
\title[On retract varieties of algebras]{On retract varieties of algebras}

\author[E. Halu\v skov\'a\ \and D. Jakub\'ikov\'a-Studenovsk\'a] %
{Em\'ilia Halu\v skov\'a* \and Danica Jakub\'ikov\'a-Studenovsk\'a**}

\newcommand{\acr}{\newline\indent}

\address{\llap{*\,}Slovak Academy of Sciences\acr
                   Mathematical Institute\acr
                   Gre\v s\'akova 5\acr
                   04101 Ko\v sice\acr
                   SLOVAKIA}
\email{ehaluska@saske.sk}

\address{\llap{**\,}P.J.\v Saf\'arik University\acr
                   Institute of Mathematics\acr
                   Jesenn\'a 5\acr
                   04154 Ko\v sice\acr
                   SLOVAKIA}
\email{danica.studenovska@upjs.sk}

\thanks{This work was supported by VEGA grants 1/0152/22 and
 2/0104/24.}



\begin{abstract}
A  retract variety is defined as a class of algebras closed under isomorphisms, retracts and products.
Let a principal retract variety be generated by one algebra and a set-principal retract variety be generated by some set of algebras.
It is shown that 
(a) not each set-principal retract variety is principal, and 
(b) not each retract variety is set-principal.
A class of connected monounary algebras $\mathcal{S}$ such that 
every retract variety of monounary algebras is generated by algebras that have all connected components from $\mathcal{S}$  and at most two connected components are isomorphic is defined, this generating class is constructively described.
All set-principal retract varieties of monounary algebras are characterized
via degree function of monounary algebras.

\end{abstract}

\maketitle

\section{Introduction}

In many branches of mathematics, classes of structures were investigated, the classes being determined by whether they are closed with respect to certain operators.
Among the first there were the operators 
$\mathbf{H}$ (homomorphic images),
$\mathbf{S}$ (subalgebras),
$\mathbf{P}$ (direct products),
creating the class named variety, for recent papers see, e.g.,
\cite{Ad,DZ,Mas,ZP}.
For years, several modifications and extensions of operators have been studied and we recall a only two:
quasivariety (closed under isomorphisms, $\mathbf{P}$, $\mathbf{S}$ and ultraproducts),
pseudovariety (closed under $\mathbf{H}$, $\mathbf{S}$ and finite direct products). 

In 1981, Duffus and Rival \cite{DR} introduced the notion of order variety as a class of posets closed with respect to isomorphism, direct products and retracts.
From papers in this direction let us mention \cite{DL,K,W,Z}.
The notion of retract connects homomorphisms and subalgebras in some sense
(see, e.g., \cite{Cz,M,Ro,JP}).

For algebras,the notion of retract variety was introduced in \cite{jak} analogously as for posets: 
it is a class of algebras closed under operators $\mathbf{P}$ and $\mathbf{R}$ (= forming retracts and their isomorphic copies).
A retract variety generated by a class of algebras $\mathcal{K}$ is equal to 
$\mathbf{RP}(\mathcal{K})$.
Retract varieties were studied for lattice-ordered groups \cite{jak} and
monounary algebras \cite{stu1,stu4}.

We will concentrate to retract varieties of monounary algebras.
Monounary algebras are very well represented by oriented graphs.
In \cite{stu1} it was shown that the system $\mathcal{L}$ of all retract varieties  
of monounary algebras forms a proper class (i.e., it is not a set).
Moreover, ordering $\mathcal{L}$ by inclusion, there exist a chain and an antichain in 
$\mathcal{L}$ which are proper classes.

Notice that each variety of algebras is principal (= generated by one algebra).
For quasivarieties this fails to hold \cite{MRW}.
We will show that this is not valid for retract varieties as well.
Dealing with set-principal (= generated by \emph{a set} of algebras) retract varieties,
there appear questions as
\begin{itemize}
\item
Is each retract variety set-principal?
\item
Is each set-principal retract variety principal?
\item
Is a system of all set-principal retract varieties a set?
\end{itemize}
The aim of this paper is a description of set-principal retract varieties of monounary algebras.
It is done in Theorem~\ref{(th2)} by using a function that is assigned to a monounary algebra as degree function, cf.~\cite{JaStPo} or as grade function, cf. ~\cite{chkn}.
As a particular results, not directly involved in the main ones, we have defined a class
$\mathcal{S}$ of some reduced connected monounary algebras such that
\begin{enumerate}
\item 
each retract variety generated by a class of connected monounary algebras is generated by a subclass of $\mathcal{S}$, see Corollary~\ref{S}.
\item
each retract variety is generated by algebras that have all connected components from $\mathcal{S}$  and at most two connected components are isomorphic, 
see Corollary~\ref{cor 5.0.2}.
\end{enumerate}

\section{Preliminaries}

The set of all positive integers will be denoted by $\mathbb{N}$;
the set of all integers will be denoted by $\mathbb{Z}$.
The cardinality of a set $A$ will be denoted by $\card A$
and the class of all ordinals by $\ord$.
We use braces $\{ \}$ for collections of elements, not exclusively  for sets.

We will apply notations and definitions concerning (partial) monounary algebras from \cite{FaSt,JaStPo,RaTh}; let us recall some of them.

A partial monounary algebra is a pair $(A,f)$, where $A$ is a nonempty set 
and $f$ is a partial unary operation on $A$.
If domain of $f$ is equal to $A$, then $(A,f)$ is called a monounary algebra.

 A  partial monounary algebra $A=(A,f)$ is called \emph{connected} if 
\begin{center}
 for every 
$x,y\in A$ there exist $m,n\in \mathbb{N} \cup \{ 0\}$ such that \\
$f^m(x),f^n(y)$ are defined and $f^m(x)=f^n(y)$.
\end{center}
 
 \begin{notation}
\label{notation T}
{\rm{
Let us denote the class of all monounary algebras by $\mathcal{U}$ and 
 the class of all connected monounary algebras by $\mathcal{U}_c$. 
 
 Further, let 
 $\mathcal{T}$ be a class of all partial monounary algebras $(A,f)$ such that
\begin{enumerate}
\item
$(A,f)$ is connected partial monounary algebra and
\item
$\card (A\setminus \dom f)= 1.$
\end{enumerate}
If $(A,f)\in \mathcal{T}$, then we denote by $c_A$ the unique element of $A\setminus \dom f$.
}}
\end{notation}

Let $(A,f)$ be a partial monounary algebra.
We will also use $A$ for this algebra, without distinguishing between the set and the algebra, if no misunderstanding can occur.

If $x\in A$, put
$$f^{-1}(x)=\{y\in A \colon y\in \dom f\ {\rm{and\ }}f(y)=x\}$$
and by induction for $n\in \mathbb{N}, n>1$
$$f^{-n}(x)=\bigcup _{z\in f^{-(n-1)}}f^{-1}(z).$$
Further, put
$$
P(x)=\{ x\} \cup \bigcup\limits_{n\in\mathbb{N}}f^{-n}(x).
$$

It is obvious that $(P(x),f\! \upharpoonleft \!\! P(x))$ is a connected partial monounary algebra and the relationship
$\card (P(x)\setminus \dom f)\leq 1$ is valid.
Moreover,
$(P(x),f\! \upharpoonleft \!\! P(x)))\in \mathcal{T}$ if and only if $x$ is not cyclic element of $(A,f)$.

\begin{lemma}
\label{lem 2.1}
Let $(A,f)$ be a partial monounary algebra, $a,b\in A$,
$ a\neq b$ and  $f(a)=a=f(b)$. 
Then the algebra 
$(P(a),f\! \upharpoonleft \!\! P(a))$  is not isomorphic to $(P(b), f\upharpoonleft P(b))$.
\end{lemma}
\begin{proof}
The class $\mathcal{T}$ is closed under isomorphisms.
We have $(P(a),f\! \upharpoonleft \!\! P(a))\notin \mathcal{T}$ and 
$(P(b),f\! \upharpoonleft \!\! P(b))\in \mathcal{T}$.
\end{proof}
\begin{notation}
\label{notation 2.1}
{\rm{
We define the following condition ($\bigstar$):

 if 
  $f(x_{1})=f(x_{2})=f(x_{3})$ and algebras 
  $(P(x_1),f\! \upharpoonleft \!\! P(x_1)),(P(x_2),f\! \upharpoonleft \!\! P(x_2))$ and
   
  $(P(x_3),f\!\upharpoonleft \!\! P(x_3))$ are isomorphic,
 then $\card \{x_{1},x_{2},x_{3}\}\leqq 2$.

We denote

$\mathcal{T}^{\bigstar}=\{ (A,f)\in \mathcal{T} \colon$ the condition ($\bigstar$) is satisfied for all $x_1,x_2,x_3\in A\}$.
}}
\end{notation}

Let us remark  that if $(A,f)\in \mathcal{U}$ and $x_1,x_2,x_3\in A$ are such that 
$f(x_{1})=f(x_{2})=f(x_{3})$ and partial algebras 
  $(P(x_1),f\! \upharpoonleft \!\! P(x_1)),(P(x_2),f\! \upharpoonleft \!\! P(x_2))$ 
    $(P(x_3),f\!\upharpoonleft \!\! P(x_3))$ are isomorphic, 
    then  none of these three elements is a fixed element of $(A,f)$ 
    according to Lemma~\ref{lem 2.1}.
    
   We will substantially use the notion of the degree of an element $x\in A$, where 
$(A,f)$ is a (partial) monounary algebra; cf. e.g. \cite{JaStPo}.
  The degree of $x$ is an ordinal or the symbol $\infty$ and it is denoted by $\s_{f}(x)$.
  It is - roughly speaking - an expression of how far we can go back from a point $x$ in the graph of algebra $A$.
  We remind the definition.
\begin{notation}
{\rm{
Let us denote by $A^{(\infty)}$  the system of all elements $x\in A$ such that there exists a sequence $\{x_{n}\}_{n\in\mathbb N\cup\{0\}}$ of elements belonging to $A$ with the property $x_0=x$ and $x_n\in \dom f, f(x_n)=x_{n-1}$ 
for each $n\in \mathbb N$. Further, we put 
$$A^{(0)}=\{x\in A\colon f^{-1}(x)=\emptyset\}.$$ 
Now we define a set $A^{(\lambda)}\subseteq A$ for each ordinal $\lambda>0$ by induction. Assume that we have defined $A^{(\alpha)}$ for each ordinal $\alpha<\lambda$. Then we put
\begin{eqnarray*}
A^{(\lambda)}=\{x\in A-\bigcup\limits_{\alpha<\lambda}A^{(\alpha)} \colon f^{-1}(x)\subseteq \bigcup\limits_{\alpha<\lambda}A^{(\alpha)}\}.
\end{eqnarray*}
The sets $A^{(\lambda)}$ are pairwise disjoint. For each $x\in A$, either $x\in A^{(\infty)}$ or there is an ordinal $\lambda$ with $x\in A^{(\lambda)}$. 
We set
$$
\s_{f}(x)=
\begin{cases}
\lambda &  \text{if there is}\ 
\lambda \in \ord \text{such that}\ x\in A^{(\lambda)},\\
\infty & \text{otherwise}.
\end{cases}
$$
We put $\lambda<\infty$ for every ordinal $\lambda$.
}}
\end{notation}
Let us remark that for every $\lambda \in \ord$ there exists an algebra 
$(A,f)\in \mathcal{U}$ such that $s_f(x)=\lambda$ for some $x\in A$.

We will use the following notation:
\begin{notation}
\label{notation 2.3}
{\rm{
Let $(A,f)\in \mathcal{U}$. 
We denote
$$A^{'}=\{x\in A\setminus A^{(\infty)} \colon 
f(x)\in A^{(\infty)}\}.$$

Let $\beta\in\ord$. Denote 

\begin{center}
\begin{tabular}{rcl}
$\mathcal{M}_{\beta}$&$=$&$\{ (A,f)\in \mathcal{U} \colon \s_{f}(x)\leqq\beta$ for each
$x\in A\setminus A^{(\infty)}\},$\\
$\mathcal{T}_{\beta}$&$=$&$\{ (A,f)\in \mathcal{T}^{\bigstar}\colon 0<\s_{f}(c_A)\leqq\beta \}$.
\end{tabular}
\end{center}
}}
\end{notation}

\subsection{Retracts and special algebras}

A nonempty subset $M$ of $A$ is said to be \textit{retract} of $(A,f)\in \mathcal{U}$ if there is a mapping $h\colon A\to M$ such that $h$ is an endomorphism of $(A,f)$ and $h(x)=x$ for each $x\in M$. 
The mapping $h$ is then called a retraction endomorphism corresponding to the retract $M$.

The following theorem proved in \cite{stu2} is essentially applied in several proofs. 
\begin{theorem} 
\label{Thm}
Let $(A,f)\in \mathcal{U}$ and $(M,f)$ be a subalgebra of $(A,f)$. 
Then $M$ is a retract of $(A,f)$ if and only if the following conditions are satisfied:
\begin{enumerate}
\item[\rm (a)] If $y\in f^{-1}(M)$, then there is $z\in M$ such that 
\begin{center}
$f(y)=f(z)$ and $\s_{f}(y)\leq\s_{f}(z)$.
\end{center}
\item[\rm (b)] For any connected component $K$ of $(A,f)$ with $K\cap M=\emptyset$, the following conditions are satisfied.
\begin{enumerate}
\item[\rm (b1)] If $K$ contains a cycle with $d$ elements, then there is a connected component $K^{*}$ of $(A,f)$ with $K^{*}\cap M\neq \emptyset$ and there is $n\in\mathbb{N}$ such that $n|d$ and $K^{*}$ has a cycle with $n$ elements.
\item[\rm (b2)] If $K$ contains no cycle and $x_{0}$ is a fixed element of $K$, then there is $y_{0}\in M$ such that 
$\s_{f}(f^{k}(x_{0}))\leq \s_{f}(f^{k}(y_{0}))$ for each 
$k\in \mathbb{N} \cup {0}$.
\end{enumerate}
\end{enumerate}
\end{theorem}

We name several meaningful monounary algebras for this paper in the next notation.

\begin{notation}
{\rm{
Denote by $Z=(\mathbb{Z}, f)$ a monounary algebra such that $f(k)=k+1$ for all $k\in \mathbb Z$. 
For $n\in \mathbb{N}$ let $\mathbb{Z}_{n}$ be the set of all integers modulo $n$ and $\underline{n}=(\mathbb Z_{n}, f)$ be a monounary algebra  such that 
if $k\in \mathbb{Z}_{n}$ then $f(k) = k+1\ (\textrm{mod}\ n)$. 

Let us define the monounary algebra $(E,f)$ as follows:

\begin{center}
\begin{tabular}{l}
$E=\mathbb{Z}\cup\{(k,1)\colon k\in \mathbb{N}\}$,\\
$f(k)=k+1\ \ \text{for each}\ \ k\in \mathbb{Z}$,\\
$
f((k,1))=
\begin{cases}
(k-1,1) &  \text{if}\ \ k\in \mathbb{N} \ \, k>1,\\
0& \text{if}\ \ k=1.
\end{cases}
$
\end{tabular}
\end{center}

For $n\in \mathbb{N}$ we define a monounary algebra $\widehat{n}=(\widehat{n},f)$ by putting

\begin{center}
\begin{tabular}{l}
$
\widehat{n}=Z_{n}\cup\{(n,1):n\in \mathbb{N}\},
$\\
$
f(k)= k+1 (\textrm{mod}\ n)\ \ \text{for each}\ \ k\in \mathbb{Z}_{n},
$\\
$
f((k,1))=
\begin{cases}
(k-1,1) &  \text{if}\ \ k\in \mathbb{N} \ \, k>1,\\
0& \text{if}\ \ k=1.
\end{cases}
$
\end{tabular}
\end{center}
}}
\end{notation}

\subsection{Several facts about retract varieties}

\begin{notation}
{\rm{
Let $\mathcal{K}$ be a class of algebras of the same type.
The retract variety generated by $\mathcal{K}$ is denoted by $\mathbf{V}(\mathcal{K})$. 
The class of algebras whose elements are all retracts (all products) of members of $\mathcal{K}$ and their isomorphic images is denoted by $\mathbf{R}(\mathcal{K})$ ($\mathbf{P}(\mathcal{K})$).

If $n\in \mathbb N$, $A_{1},\dots, A_{n}$ be algebras of the same signature, then we write $\mathbf{V}(A_{1}\dots, A_{n})$ instead of $\mathbf{V}(\{A_{1}\dots, A_{n}\})$.

For a (partial) algebra $A$ we denote by $[A]$ the class all $B$ isomorphic to $A$.
If $B\in [A]$, then we write $B\cong A$.
}}
\end{notation}

Proposition 1.3 in \cite{stu4} says that 
$\mathbf{V}(\mathcal{K}) = \mathbf{RP}(\mathcal{K})$ for each $\mathcal{K}\subseteq \mathcal{U}$.
The proof of this Proposition works generally, therefore we obtain
\begin{lemma}
Let $\mathcal{K}$ be a class of algebras of the same type.
Then $\mathbf{V}(\mathcal{K}) = \mathbf{RP}(\mathcal{K})$.
\end{lemma}

\begin{lemma}
\label{pom7}
Let $\mathcal{K},\mathcal{L}$  be a classes of algebras of the same type.
If $A\in \mathbf{V}(\mathcal{L})$ for each $A\in \mathcal{K}$,
then $\mathbf{V}(\mathcal{K})\subseteq \mathbf{V}(\mathcal{L})$.
\end{lemma}
\begin{proof}
Let $B\in \mathbf{V}(\mathcal{K})=\mathbf{RP}(\mathcal{K})$.
Then there exist a set $I$ and $A_i\in \mathcal{K}, i\in I$ such that $B$ is a retract of 
$\prod _{i\in I}A_i$.
The assumption of Lemma says that $A_i\in \mathbf{V}(\mathcal{L})$ for each $i\in I$.
It yields
$$B\in \mathbf{RPV}(\mathcal{L})=\mathbf{V}(\mathcal{L}).$$
\end{proof}

The following two  properties of retract varieties follow from definitions.
\begin{lemma}
\label{pom2}
Let $A$ be an algebra. 
If $\mathcal{V}$ is a retract variety and $A\in \mathcal{V}$,
then $\mathbf{V}(A)\subseteq \mathcal{V}$.
\end{lemma}

\begin{lemma}
\label{pom1}
$\mathbf{V}(\underline{1})=[\underline{1}]$ and if $n\in \mathbb N, n>1$, then
$\mathbf{V}(\underline{n})$ is the class of all monounary algebras which have every component isomorphic to $\underline{n}$.
\end{lemma}

\begin{proposition}
\label{lemma 2.1}
The retract variety $\mathbf{V}(\underline{2}, \underline{3})$
is not principal.
\end{proposition}
\begin{proof}
We will prove two statements:
\begin{enumerate}
\item
$\mathbf{V}(\underline{2},\underline{3})=\mathbf{V}(\underline{2})\cup
\mathbf{V}(\underline{3})\cup \mathbf{V}(\underline{6})$
\item
If $A\in \mathcal{U}$ is such that 
$\mathbf{V}(A)\subseteq \mathbf{V}(\underline{2},\underline{3})$, then $\mathbf{V}(A)\in\{\mathbf{V}(\underline{2}), \mathbf{V}(\underline{3}), \mathbf{V}(\underline{6})\}$.
\end{enumerate}

Let $B\in \mathbf{V}(\underline{2},\underline{3})$. 
We have $\mathbf{V}(\underline{2},\underline{3})= \mathbf{RP}(\underline{2},\underline{3})$, 
thus there are sets $I, J$ and monounary algebras $A_{i}$ 
for each $i\in I$, $B_{j}$ for each $j\in J$ such that
 $A_{i}\cong\underline{2}$ for each $i\in I$,
 $B_{j}\cong\underline{3}$ for each $j\in J$ and
 $B$ is isomorphic to some retract $M$ of 
 $\prod\limits_{i\in I}A_{i}\times\prod\limits_{j\in J}B_{j}$.

If $I= \emptyset$, then $M$ is a retract of 
$\prod\limits_{j\in J}B_{j}$, 
hence  $B\in \mathbf{V}(\underline{3})$.  
Analogously, if $J=\emptyset$, then $B\in \mathbf{V}(\underline{2})$.

Let $I\neq\emptyset\neq J$. Then $\prod\limits_{i\in I}A_{i}\times\prod\limits_{j\in J}B_{j}$ is a disjoint union of cycles with cardinality $6$, hence $M$ is a disjoint union of $6$-element cycles, therefore $B\in \mathbf{V}(\underline{6})$. 

If $n\in \{ 2,3\}$, then $\mathbf{V}(\underline{n})\subset \mathbf{V}(\underline{2},\underline{3})$.
Suppose that $A\in \mathbf{V}(\underline{6})$. 
For each $a\in A$ let $D_{a}=\underline{2}$ and $E_{a}=\underline{3}$.
Put $D= \prod\limits_{a\in A}D_{a}\times\prod\limits_{a\in A}E_{a} $. 
Then $D$ is a disjoint union of at least $\card A$ cycles with cardinality $6$. 
We have $A\in \mathbf{R}(D)$ according to Theorem~\ref{Thm} and therefore 
$A\in \mathbf{V}(\underline{2},\underline{3})$.

It yields 
the statement (1).

Suppose that $\mathbf{V}(A)\subseteq \mathbf{V}(\underline{2},\underline{3})$.
Then $A\in \mathbf{V}(\underline{2},\underline{3})$ and thus  
$A\in \mathbf{V}(\underline{n})$ for some $n\in \{2,3,6\}$.
Therefore  $\mathbf{V}(A)\subseteq \mathbf{V}(\underline{n})$ according to Lemma~\ref{pom2}.
Further, $\underline{n}\in \mathbf{R}(A)$.
This implies
 $\mathbf{V}(\underline{n})\subseteq \mathbf{V}(A)$ according to Lemma~\ref{pom7}.
\end{proof}

\section{Auxiliary results}

The definition of degree of an element yields
\begin{lemma}
\label{pom3}
Let $(A,f)\in \mathcal{U}$ and $B\in \mathbf{V}(A)$.
If $A^{(\infty )}=A$, then $B^{(\infty )}=B$.
\end{lemma}

Let us remark that direct products preserve  injectivity of operations. Therefore
\begin{lemma}
\label{pom5}
Let $(A,f)\in \mathcal{U}$ and $B\in \mathbf{V}(A)$.
If
the operation $f$ is injective on $A$,
then
the operation $f$ is injective on $B$.
\end{lemma}

\begin{lemma}
\label{pom6}
Let $(A,f)\in \mathcal{U}$ and $B\in \mathbf{V}(A)$.
If the operation $f$ is injective on 
$A\setminus \bigcup _{a\in A'}P(a)$, then
the operation $f$ is injective on 
$B\setminus \bigcup _{b\in B'}P(b)$.
\end{lemma}
\begin{proof}
Suppose that $B\in \mathbf{V}(A)$ and
\begin{center}
$B\in \mathbf{R}(D)$, where $D=\prod _{i\in I} A_i, A_i\cong A$ for all $i\in I$.
\end{center} 
An element $d=(a_i, i\in I)\in D^{(\infty)}$ if and only if $a_i \in A^{(\infty)}$ for each $i\in I$.
Therefore the operation of the algebra 
$D\setminus \bigcup _{d\in D'}P(d)$ is injective and the operation of
$B\setminus \bigcup _{b\in B'}P(b)$ is injective, too, since it is a subalgebra of $D\setminus \bigcup _{d\in D'}P(d)$.
\end{proof}

\begin{proposition}
\label{prop_SP}
Let $\mathcal{V}$ be a set-principal retract variety of monounary algebras. Then there is $\beta\in \ord$ such that 
$\mathcal{V}\subseteq \mathcal{M}_{\beta}$.
\end{proposition}
\begin{proof}
Suppose that $\mathcal{V}=\mathbf{V}(\{A_{i}\colon i\in I\})$, where $I$ is a set,
$A_i\in \mathcal{U}$ for each $i\in I$. 
Let $$S=\{\s_{f}(x)\colon x\in A_{i}, i\in I\} \setminus \{\infty\}.$$
There exists $\beta\in \ord$ such that 
\begin{enumerate}
\item $\beta\geqq \alpha$ for each $\alpha\in S$.
\end{enumerate}
Assume that $A\in \mathcal{V}$ and 
$D\in \mathbf{P}(\{A_{i}\colon i\in I\})$ is such that $A\in \mathbf{R}(D)$.
Let 
\begin{center}
$D=\prod\limits_{j\in J}B_{j}$, where 
$B_j \in \{A_{i}\colon i\in I\}$
for each $j\in J$.
\end{center}
Take $x\in D\setminus D^{(\infty)}$. 
If $j\in J$, then $\s_{f}(x(j))\in S\cup\{\infty\}$ and there is $j_{0}\in J$ with $\s_{f}(x(j_{0}))\neq\infty$, therefore (1) yields
 $\s_{f}(x(j_{0}))\leqq\beta$.
From this we obtain
 $\s_{f}(x)\leqq\beta$.
Therefore $A\in \mathcal{M}_{\beta}$ since $A\subseteq D$.
\end{proof}

\subsection{On classes $\mathcal{T}_{\beta}$}

Now we will deal with the class $\mathcal{T}$. 
If $(A,f)\in \mathcal{T}$, then $(A,f)$ is connected and contains 
the unique element $c_A$ such that the operation $f$ is not defined in it, see Notation~\ref{notation T}.
Remind that classes $\mathcal{T}_{\beta}$ for $\beta \in \ord$ were introduced in Notation~\ref{notation 2.3}.

\begin{lemma}
\label{lemma 3.1}
Let $n\in \mathbb{N}$. Then $\mathcal{T}_{n}/\cong$ is a finite set.
\end{lemma}
\begin{proof}
Let us prove the assertion by induction.
\par
(I) Let $n=1$. Suppose that $(A,f)\in \mathcal{T}_{1}$. Then Notation~\ref{notation 2.3} implies that $A=\{c_{A}\}\cup f^{-1}(c_{A})$. 
According to 
($\bigstar$)
we obtain that $\card f^{-1}(c_{A})\leqq 2$. Thus $\card (\mathcal{T}_{1}/\cong)=2$.
\newline\par
(II) Let $n\in\mathbb{N}$, $n>1$ and suppose that $\mathcal{T}_{m}/\cong$ is a finite set for each $m\in\mathbb{N}$, $m < n$. Let $(A,f)\in \mathcal{T}_{n}$. By the definition of $\s_{f}$ we get
\begin{enumerate}
\item [(1)] if $y\in f^{-1}(c_{A})$, then $\s_{f}(y)\in \{m\in\mathbb{N} \colon m < n\}$.
\end{enumerate}
Hence,
\begin{enumerate}
\item [(2)] if $y\in f^{-1}(c_{A})$, then $(P(y),f)$ belongs to $\mathcal{T}_{m}$ for some $m\in\mathbb{N}$, $m < n$.
\end{enumerate}
Denote, for $m\in\mathbb{N}$, $m < n$,
\begin{align*}
A_{m}=\{(P(y),f) \colon y\in f^{-1}(c_{A}), \s_{f}(y)=m\}.
\end{align*}
By ($\bigstar$) and by the induction assumption, there are finitely many possibilities for $A_{m}$ ($m<n$), thus there are only finitely many non-isomorphic $(A,f)$ in $\mathcal{T}_{n}$.
\end{proof}

\begin{notation}
\label{tau-beta-notation}
Let $\beta\in\ord$. We define $\tau(\beta)$ by induction as follows:
\begin{enumerate}
\item[\rm{(i)}] if $\beta$ is finite, then $\tau(\beta)=\aleph_{0}$; 
\item[\rm{(ii)}] if $\beta$ is infinite, then $\tau(\beta)=\sup\{2^{\tau(\alpha)} \colon \alpha\in\ord, \alpha<\beta\}$.
\end{enumerate}
\end{notation}

\begin{lemma}
\label{tau-beta-lemma}
Let $\beta\in\ord$. Then $\mathcal{T}_{\beta}/\cong$ is a set and $\card(\mathcal{T}_{\beta}/\cong)\leqq \tau(\beta)$.
\end{lemma}
\begin{proof}
By induction:
\par
(I) If $n\in\mathbb{N}$, then 
$\card(\mathcal{T}_{n}/\cong)< \aleph_{0} = \tau(n)$ according to Lemma~\ref{lemma 3.1}.
\newline\par
(II) Let $\beta$ be an infinite ordinal and suppose that for each $\alpha\in\ord$, $\alpha < \beta$ the assertion is valid. Consider $(A,f)\in\mathcal{T}_{\beta}$. The definition of $\s_{f}$ implies
\begin{enumerate}
\item[(1)] if $y\in f^{-1}(c_{A})$, then $\s_{f}(y)\in \{\alpha\in\ord \colon \alpha < \beta\}$.
\end{enumerate}
Hence,
\begin{enumerate}
\item[(2)] if $y\in f^{-1}(c_{A})$, then $(P(y),f)$ belongs to $\mathcal{T}_{\alpha}$ for some $\alpha < \beta$.
\end{enumerate}
The condition ($\bigstar$) is valid in $(A,f)$, thus  there are 
$\prod\limits_{\alpha<\beta}2\cdot (2^{\card(\mathcal{T}_{\alpha}/\cong)})$ 
 possibilities how $(A,f)$ can look like.
 It yields that $\mathcal{T}_{\beta}/\cong$ is a set and
 \begin{align*}
\card(\mathcal{T}_{\beta}/\cong)=
\prod\limits_{\alpha<\beta}2\cdot (2^{\card(\mathcal{T}_{\alpha}/\cong)})\leqq \prod\limits_{\alpha<\beta}2^{\tau(\alpha)} = \tau(\beta).
\end{align*}
\end{proof}

\section{Connected algebras}

\begin{notation}
\label{S notation}
{\rm {
Consider the following sets of connected monounary algebras:
}}

$\mathcal{U}^{\bigstar}_c=\{ (A,f)\in \mathcal{U}_c \colon$ the condition ($\bigstar$) is satisfied for all $x_1,x_2,x_3\in A\}$,

\vspace{0.1cm}
$\mathcal{S}^{(0)}=[ Z,E]\cup \bigcup _{n\in \mathbb{N}} [\underline{n},\widehat{n}]$,

$\mathcal{S}^{(1)}=\{ A\in \mathcal{U}^{\bigstar}_c \colon$ there is $a\in A$
such that $A^{'}=\{a\}, A\setminus P(a)\in \mathcal{S}^{(0)}\}$,
 
$\mathcal{S}^{(2)}=\{ A\in \mathcal{U}^{\bigstar}_c\colon A^{(\infty)}=\emptyset \}$.

{\rm{Put}}
$$\mathcal{S}=\mathcal{S}^{(0)}\cup \mathcal{S}^{(1)}\cup \mathcal{S}^{(2)}.$$
\end{notation}

We remark that 
$\mathcal{S}^{(0)}\subset \mathcal{U}^{\bigstar}_c$. 
\begin{lemma}
\label{lemma 4.1}
Let $A\in \mathcal{U}_c$ possess no cycle.\\
If $A\ncong Z$ and $A^{(\infty)}=A$, then $\mathbf{V}(A)=\mathbf{V}(E)$.
\end{lemma}
\begin{proof}
By Theorem~\ref{Thm}, $E\in \mathbf{R}(A)$. 
Thus $\mathbf{V}(E)\subseteq \mathbf{V}(A)$ according to Lemma~\ref{pom7}.

To see the opposite inclusion we need to prove $A\in \mathbf{V}(E)$.
It holds  trivially in the case $A\cong E$. 
Suppose that $A\ncong E$. 
Put $\alpha=\card A$. 
Let $I_{j}$ for $j\in \mathbb{Z}$ be disjoint sets of indices such that 
$\card I_{j} = \alpha$ for each $j\in Z$ and $I=\bigcup\limits_{j\in \mathbb{Z}} I_{j}$. Put $B_{i}=E$ for each $i\in I$ and $B=\prod\limits_{i\in I}B_{i}$. 
Denote by $K$ the connected component of $B$ such that $K$ contains the element $b\in B$ with $b(i)=j$ for each $i\in I_{j}$, $j\in \mathbb{Z}$.

Let $x\in K$.
Then there are $k,l\in \mathbb{N}$ such that $f^k(x)=f^l(b)$ since $K$ is connected.
Denote $d=f^l(b)$.
We have $d(i)=j$ for each $j\in \mathbb{N}$ and $i\in I_{-l+j}$.
Therefore $d(i)=k$ for every $i\in I_{-l+k}$.
Consequently $(f(x))(i)=1$ and $x(i)=0$ since $f^{-1}(1)=\{0\}$ in the algebra $E$.
We get  $\card f^{-1}(x)\geq\alpha$. 

Thus $K$ is a connected algebra such that $\s_{f}(x)=\infty$ and $\card f^{-1}(x)\geq\alpha$ for each $x\in K$. 
Such an algebra contains obviously a subalgebra $T$ such that $T\cong A$. 
Further, $T$ is a retract of $B$ according to Theorem~\ref{Thm} and the fact that no connected component of $B$ contains a cycle.
Hence $A\in \mathbf{RP}(E)=\mathbf{V}(E)$. 
\end{proof}

\begin{lemma}
\label{lemma 4.2}
Let $n\in \mathbb{N}$ and $A\in \mathcal{U}_c$ possess $n$-element cycle.\\ 
If $A\ncong \underline{n}$ and $A^{(\infty)}=A$, then
$\mathbf{V}(A)=\mathbf{V}(\widehat{n})$.
\end{lemma}
\begin{proof}
Obviously $\widehat{n}\in \mathbf{R}(A)$ and therefore
$ \mathbf{V}(\widehat{n})\subseteq \mathbf{V}(A)$. 
To see the opposite inclusion, we need to prove that
$A\in \mathbf{V}(\widehat{n})$.

Suppose that $A$ is not isomorphic to $\widehat{n}$.
Denote $\alpha=\card A$.
Put $B_{a}=\widehat{n}$ for each $a\in A$ and $B=\prod\limits_{a\in A}B_{a}$. 
We have $B^{(\infty)}= B$ according to Lemma~\ref{pom3}.
Next, each connected component of $B$ contains a cycle with cardinality $n$ and coordinates of cyclic elements of $B$ create subsets of $\mathbb{Z}_n$. 

Consider a component $K$ of $B$ such that there is $b\in K, b$ cyclic and every element of $\mathbb{Z}_n$ occurs $\alpha $ times in $b$.
Analogously as in the proof of the previous lemma we can see that
$\card f^{-1}(x)\geq\alpha$ for each $x\in K$.
Therefore $B$ contains a subalgebra isomorphic to $A$. 
Hence, according to Theorem \ref{Thm}, $A\in \mathbf{R}(B)$. 
This implies that $A\in \mathbf{V}(\widehat{n})$.
\end{proof}

\begin{lemma}
\label{lemma 4.3}
Let $A\in \mathcal{U}_c$. 
Then there are a set $I$ and algebras $B_{i}\in \mathbf{R}(A)$ with 
$\card B_{i}^{'}=1$ for each $i\in I$ such that
$$\mathbf{V}(A)= \mathbf{V}(\{B_{i}\colon i\in I \}).$$
\end{lemma}
\begin{proof}
Let $A^{'}= \{b_{i}\colon i\in I \}$
and $0\notin I$.
We have $\card I>1$. 

Put
$$
A_{0}=A-\bigcup_{i\in I}P(b_{i})
$$
and for each  $i\in I$ put 
$$B_{i}=A_{0}\cup P(b_{i}).$$ 
We obtain $B_{i}\in \mathbf{R}(A)$ and $\card B_{i}^{'}=1$, since  
$f(b_{i})\in A^{(\infty)}$ for each  $i\in I$.
Further, $\mathbf{V}(\{B_{i}\colon i\in I \})\subseteq \mathbf{V}(A)$ according to Lemma~\ref{pom7}.

Obviously $A_{0}$ is a retract of $A$, thus there is a retraction endomorphism $\varphi$ of $A$ onto $A_{0}$. Denote $B=\prod\limits_{i\in I}B_{i}$. Put
$$
T_{0}=\{t\in B\colon (\exists a\in A_{0})(t(i)=a \ \ \text{for each}\ \ i\in I)\}
$$
\begin{center} 
$
T_{i}=\{t\in B\colon t(i)\in P(b_{i}), t(j)=\varphi(t(i)) \ \ \text{for}\ \ j\in I-\{i\}\}
$
for $i\in I$,
\end{center} 
$$T=\bigcup_{i\in I\cup\{0\}}T_{i}.$$

In view of Theorem \ref{Thm}, $T$ is a retract of $B$.

Let us define a mapping $\Phi: A\to T$ as follows.
If $a\in A^{(\infty)}$, then $\Phi (a)= t$ where $t(j)=a$ for all $j\in I$.
If $i\in I$ and $a\in P(b_i)$, then $\Phi (a)= t$ where $t(i)=a$ and $t(j)=\varphi(a)$
for each $j\in I\setminus \{ i\}$.
It can be verified that $\Phi$ is an isomorphism.
 
Therefore we get 
$
A\in \mathbf{RP}(\{B_{i}\colon i\in I\})=\mathbf{V}(\{B_{i}\colon i\in I\}).
$
\end{proof}

\begin{lemma}
\label{lemma 4.4}
Let $A\in \mathcal{U}_c$.\\
Then there exists  $B\in \mathbf{R}(A)\cap \mathcal{U}^{\bigstar}_c$ such that $\mathbf{V}(A)=\mathbf{V}(B)$.
\end{lemma}
\begin{proof}
If $A\in \mathcal{U}^{\bigstar}_c$ then  the assertion holds trivially.
Now assume that $x_1,x_2,x_3\in A$ are such that the condition ($\bigstar$) is not valid. 
For each $x\in A$ let
$$
U(x)=\{y\in A\colon f(x)=f(y), P(x)\cong P(y)\}.
$$
We have $U(x_1)=U(x_2)=U(x_3)$ and card $U(x_1)>2$.
Denote
$$
\{U_{i}\colon i\in I\}= \{U(x)\colon x\in A, \card U(x)>2\}.
$$
For each $i\in I$ take two distinct fixed elements of $U_{i}$ and denote them by $u_{1}^{i}$, $u_{2}^{i}$. 
Now put
$$
B=(A\setminus 
(\bigcup_{i\in I} \bigcup _{u\in U_{i}} P(u)))\cup 
\bigcup _{i\in I}(P(u^i _1)\cup P(u^i _2)).
$$
Then $B\in \mathbf{R}(A)\cap \mathcal{U}^{\bigstar}_c$ and 
$\mathbf{V}(B)\subseteq \mathbf{V}(A)$ according to Lemma~\ref{pom7}.

Let $\kappa=\sup\{\card U_{i} \colon i\in I\}$ and let $J$ be a set with $\card J=\kappa$. Put $B_{j}=B$ for each $j\in J$ and let $C=\prod\limits_{j\in J}B_{j}$.

We have that
for $i\in I$ the set 
$$
D_{i}=\{c\in C\colon c(j)\in \{u_{1}^{i},u_{2}^{i}\}\ \ \text{for each}\ \ j\in J\}
$$
has 
$2^{\kappa}$ elements, thus there is an injection $\xi_{i}\colon U_{i}\to D_{i}$. Further, if $i\in I $ and $x\in U_{i}$, then there exists isomorphisms
$$
\psi_{1}^{x} \colon P(x)\to P(u_{1}^{i}),
$$
$$
\psi_{2}^{x} \colon P(x)\to P(u_{2}^{i}).
$$
Now let us define an injective homomorphism $\nu \colon A\to C$.
If $a\in B$, then denote 
\begin{enumerate}
\item[(1)] $\nu(a)=\overline{a}\in C$, where $\overline{a}(j)=a$ for each $j\in J$.
\end{enumerate}
If $i\in I$, $a\in U_{i}\setminus \{u_{1}^{i},u_{2}^{i} \}$, then we define
\begin{enumerate}
\item[(2)] $\nu(a)=\xi_{i}(a)$.
\end{enumerate}
If $i\in I$, $x\in U_{i}\setminus \{u_{1}^{i},u_{2}^{i} \}$ and $a\in P(x)\setminus\{x\}$ then we define
\begin{enumerate}
\item[(3)] $\nu(a)=z$, where
\end{enumerate}
$$
z(j)=
\begin{cases}
\psi_{1}^{x}(a) &  \text{if}\ \, (\xi_{i}(x))(j)=u_1^{i},\\
\psi_{2}^{x}(a) & \text{if} \ \, (\xi_{i}(x))(j)=u_2^{i},
\end{cases}
$$ for each $j\in J$.
Denote 
$T=\nu(A).$
It is a technical matter to verify
that $T$ is a retract of $C$; hence
$$
A\in \mathbf{R}(C)\subseteq \mathbf{RP}(B)= \mathbf{V}(B).
$$
We get $\mathbf{V}(A)\subseteq \mathbf{V}(B)$.
\end{proof}

\begin{lemma}
\label{lemma 4.5}
Let $A\in \mathcal{U}_c$ and
$a\in A$ be such that 
$A^{'}=\{a\}$.  
If  $A$ possesses no cycle and and $A\setminus P(a)\ncong Z$, 
then there is $B\in \mathbf{R}(A)\cap \mathcal{S}^{(1)}$ such that  
$\mathbf{V}(A)=\mathbf{V}(B,E).$
\end{lemma}
\begin{proof}
Let $A^*$ be an algebra from Lemma \ref{lemma 4.4}.
Then $(A^*)'=\{ a\}$ according to Lemma \ref{pom3} and 
$A^*\setminus P(a)$ is not isomorphic to $Z$ according to Lemma \ref{pom6}.
Therefore  the same assumptions as for $A$ are valid for $A^*$.
Further, if $B$ is such that the statement is true for $A^*$, then $B\in \mathbf{R}(A^*)\subseteq \mathbf{R}(A)$ and
$\mathbf{V}(A)=\mathbf{V}(A^*)=\mathbf{V}(B,E)$.

So, we can suppose that $A=(A,f)\in \mathcal{U}^{\bigstar}_c$. 
Obviously $E\in \mathbf{R}(A)$.
Denote by 
 $A_{0}$ a subalgebra of $A$ such that
 $A_{0}\cong Z$ and
 $f(a)\in A_{0}$. Put
$B=A_{0}\cup P(a).$
Then
 $B\in \mathbf{R}(A)\cap \mathcal{S}^{(1)}$.
It yields $\mathbf{V}(B,E)\subseteq \mathbf{V}(A)$ according to Lemma~\ref{pom7}.

There is a retraction homomorphism $\psi$ of $A$ onto $A_{0}$. Denote
$$
T=\{t\in B\times A^{\infty} \colon (\exists x\in P(a)) (t=(x,\psi(x)))\ \ \text{or}\ \ (\exists x\in A-P(a)) (t=(\psi(x),x))\}.
$$
It is easy to see that
$T\cong A$ and $T$ is a retract of $B\times A^{(\infty)}$. Therefore 
 $A\in \mathbf{V}(B, A^{(\infty)})$.
Further, the algebra $A^{(\infty)}$ fulfills the assumptions of Lemma~\ref{lemma 4.1}, hence we obtain
 $A^{(\infty)}\in \mathbf{V}(E)$.
Thus $\mathbf{V}(B, A^{(\infty)})\subseteq \mathbf{V}(B,E)$ according to Lemma~\ref{pom7}.
\end{proof}

Analogously, we can prove
\begin{lemma}
\label{lemma 4.6}
Let $A\in \mathcal{U}_c$ and
$a\in A$ be such that 
$A^{'}=\{a\}$. 
If $A$ possesses a cycle with $n$ elements, $n\in \mathbb{N}$ 
and $A\setminus P(a)\notin \mathcal{S}^{(0)}$,
then there is $B\in \mathbf{R}(A)\cap \mathcal{S}^{(1)}$ such that  
$\mathbf{V}(A)=\mathbf{V}(B,\widehat{n})$. 
\end{lemma}

\begin{lemma}
\label{lemma 4.7}
If $A\in \mathcal{U}_c$, then there is a set $I$ and algebras $B_{i}\in \mathbf{R}(A)\cap \mathcal{S}$ for each $i\in I$ such that $\mathbf{V}(A)=\mathbf{V}(\{B_{i}\colon i\in I\})$.
\end{lemma}
\begin{proof}
If $A\in \mathcal{S}^{(0)}$, then the assertion holds. 
Let $A\notin \mathcal{S}^{(0)}$. 
By Lemma \ref{lemma 4.4} there is an algebra $B\in \mathbf{R}(A)\cap \mathcal{U}^{\bigstar}_c$ such that $$\mathbf{V}(A)=\mathbf{V}(B).$$
If $B\in \mathcal{S}$,
then the proof is finished; therefore let $B\notin \mathcal{S}$. 
Then
$B^{(\infty)}\neq\emptyset$ since $B\notin \mathcal{S}^{(2)}$.

Assume that $A$ has a cycle of length $n$, $n\in \mathbb{N}$.

Let $B=B^{(\infty)}$. 
Then Lemma
\ref{lemma 4.2} implies that $\mathbf{V}(A)=\mathbf{V}(B)=\mathbf{V}(\widehat{n})$.

Suppose that $B\neq B^{(\infty)}$. 
Then $B^{'}\neq\emptyset$ and 
Lemma~\ref{lemma 4.3} yields that there is a set $J$ and algebras $B_{j}\in \mathbf{R}(B)$ with $\card(B^{'}_{j}) =1$ for each $j\in J$ such that
$$\mathbf{V}(B)=\mathbf{V}(\{B_{j}\colon j\in J\}).$$
Assume that $j\in J$. 
Then $B_j\in \mathcal{U}^{\bigstar}_c$  according to $B_j\in \mathbf{R}(B)$ and $B\in \mathcal{U}^{\bigstar}_c$.
Further, $B'_j=\{ b_j\}$ for some $b_j\in B_j$ since $\card(B^{'}_{j}) =1$.
If $B_{j}\setminus P(b_{j})\in \mathcal{S}^{(0)}$, then $B_j \in \mathcal{S}^{(1)}$ by the definition of $\mathcal{S}^{(1)}$.
If
$B_{j}\setminus P(b_{j})\notin \mathcal{S}^{(0)}$, then in view of \ref{lemma 4.6} there is 
$C_{j}\in \mathbf{R}(B_j)\cap \mathcal{S}^{(1)}$ such that
$$\mathbf{V}(B_j)=\mathbf{V}(C_{j},\widehat{n}).$$
Put 
\begin{center}
$J_{2}=\{j\in J\colon B_{j}\in \mathcal{S}^{(1)}\}$, 
$J_{1}=J\setminus J_{2}$.
\end{center}
If $J_{1}=\emptyset$, then 
the proof is finished.

If $J_{1}\neq \emptyset$, then the set
$$M=\{B_{j}\colon j\in J_{2}\}\cup \{C_{j}\colon j\in J_{1}\}\cup \{\widehat{n}\}
\subset \mathcal{S}^{(0)}\cup \mathcal{S}^{(1)}\subset \mathcal{S}$$ 
and 
$$\mathbf{V}(\{B_{j}\colon j\in J\})=\mathbf{V}(M)$$
according to Lemma~\ref{pom7}.
This completes the proof for $A$ with a cycle. 

If $A$ has no cycle, then we can proceed analogously, use algebras $Z$, $E$ 
and Lemmas \ref{lemma 4.1}, \ref{lemma 4.3} and \ref{lemma 4.5}.
\end{proof}

\begin{corollary}
\label{S}
Let $\mathcal{K}\subseteq \mathcal{U}_c$. 
Then there exists $\mathcal{L}\subseteq \mathcal{S}$ such that $\mathbf{V}(\mathcal{L})=\mathbf{V}(\mathcal{K})$.
\end{corollary}

\subsection{Retract varieties generated by a set of connected algebras}

Let $\beta\in\ord$.
The class $\mathcal{M}_{\beta}$ is defined in Notation~\ref{notation 2.3}.
\begin{notation}
{\rm{
 Denote  

$\mathcal{S}_{\beta} =\mathcal{M}_{\beta}\cap \mathcal{S}$,
$\mathcal{S}_{\beta}^{(1)} =\mathcal{M}_{\beta}\cap \mathcal{S}^{(1)}$,
$\mathcal{S}_{\beta}^{(2)} =\mathcal{M}_{\beta}\cap \mathcal{S}^{(2)}$.
}}
\end{notation}

We remark that $\mathcal{M}_{\beta}\cap \mathcal{S}^{(0)}= \mathcal{S}^{(0)}$ for every $\beta\in\ord$. 

\begin{lemma}
Let $\beta\in\ord$. Then  
$\mathcal{S}_{\beta}^{(1)}/\cong$ is a set and 
$\card(\mathcal{S}_{\beta}^{(1)}/\cong)\leqq \tau(\beta).$ 
\end{lemma}
\begin{proof}
We define a mapping 
$$
\varphi\colon (\mathcal{S}_{\beta}^{(1)}/\cong)\to 
(\mathbb{N}\cup \{0\})\times\mathbb{Z}\times(\mathcal{T}_{\beta}/\cong)
$$
as follows. 
Let $A\in \mathcal{S}^{(1)} ,A^{'}=\{a\}$.
If $A\setminus P(a)\cong Z$, then put 
$$
\varphi([A])=(0,0,[P(a)]).
$$
If $A\setminus P(a)\cong \underline{n}$ for some $n\in \mathbb{N}$, then put 
$$
\varphi([A])=(n,0,[P(a)]).
$$
If $A\setminus P(a)\cong \widehat{n}$ for some $n\in \mathbb{N}$, then there is a uniquely determined $k\in \mathbb{N}\cup\{0\}$ such that $f^{k}(a)$ does not belong to a cycle, $f^{k+1}(a)$ belongs to a cycle; put 
$$
\varphi([A])=(n,k+1,[P(a)]).
$$
The mapping $\varphi$ is injective, therefore $\mathcal{S}_{\beta}^{(1)}/\cong$ is a set and Lemma~\ref{tau-beta-lemma} implies
$$
\card(\mathcal{S}_{\beta}^{(1)}/\cong)\leqq \aleph_{0}\times\aleph_{0}\times \tau(\beta)=\tau(\beta).
$$
\end{proof}

\begin{lemma}
Let $\beta\in\ord$. Then  
$\mathcal{S}_{\beta}^{2}/\cong$ is a set and 
$\card(\mathcal{S}_{\beta}^{(2)}/\cong)\leqq (\tau(\beta))^{\aleph_{0}}.$ 
\end{lemma}
\begin{proof}
For a class $[A]\in \mathcal{S}_{\beta}^{(2)}/\cong$ take a fixed representant $\overline{A}$ of this class and let $a$ be an arbitrary (fixed) element of $\overline{A}$. We define a mapping 
$$
\varphi\colon (\mathcal{S}_{\beta}^{(2)}/\cong)\to (\mathcal{T}_{\beta}/\cong)^{\mathbb N}
$$
as follows. If $[A]\in \mathcal{S}_{\beta}^{(2)}/\cong$, then put
$$
\varphi([A])=(P(a),P(f(a)),P(f^{2}(a)),\dots ).
$$
The mapping $\varphi$ is injective, therefore $\mathcal{S}_{\beta}^{7}/\cong$ is a set and we obtain in view of Lemma~\ref{tau-beta-lemma},
$$
\card(\mathcal{S}_{\beta}^{(2)}/\cong)\leqq (\tau(\beta))^{\aleph_{0}}.
$$
\end{proof}

\begin{corollary}
\label{corollary 4.0.1}
Let $\beta\in \ord$. Then 
$\mathcal{S}_{\beta}/\cong$ is a set and
$\card(\mathcal{S}_{\beta}/\cong)\leqq (\tau(\beta))^{\aleph_{0}}$. 

\end{corollary}

\begin{proposition}
Let $\mathcal{K}\subseteq \mathcal{U}_c$. 
The following conditions are equivalent:
\begin{enumerate}
\item[\rm{(i)}] $\mathbf{V}(\mathcal{K})$ is set-principal,
\item[\rm{(ii)}] there is $\beta\in\ord$ such that 
$\mathcal{K}\subseteq \mathcal{M}_{\beta}$.
\end{enumerate} 
\end{proposition}
\begin{proof}
Let (ii) hold. If $A\in\mathcal{K}$, then Lemma~\ref{lemma 4.7} implies that there are a set $I_{A}$ and algebras 
$B_{i}^{A}\in \mathbf{R}(A)\cap \mathcal{S}$ for $i\in I_{A}$ such that
$A\in \mathbf{V}(\{B_{i}^{A}\colon i\in I_{A}\}) $. 
 
We get 
$$
\mathbf{V}(\mathcal{K})=\mathbf{V}(\{A\colon A\in\mathcal{K}\})=\mathbf{V}(\bigcup\limits_{A\in\mathcal{K}}\{B_{i}^{A}\colon i\in I_{A}\}).
$$
Further, $B_{i}^{A}\in \mathcal{S}_{\beta}$ since $B_{i}^{A}\in \mathbf{R}(A)$ and $A\in \mathcal{M}_{\beta}$.
That means that $\mathbf{V}(\mathcal{K})$
is set-principal with respect to Corollary~\ref{corollary 4.0.1}.

The opposite implication is proved in Proposition~\ref{prop_SP}.
\end{proof}

\section{General case}

In this section we finish a description of set principal retract varieties of monounary algebras.

For a monounary algebra $A$ consider the following condition:
\begin{enumerate}
\item[{$(^{\bigstar} _{\bigstar})$}] if $C_{1}$, $C_{2}$, $C_{3}$ are connected components of $A$ such that $C_{1}\cong C_{2}\cong C_{3}$, then $\card \{C_{1},C_{2},C_{3}\}\leqq 2$.
\end{enumerate}

We denote 
\begin{center}
$\mathcal{U}^{\bigstar}_{\bigstar}=\{ (A\in \mathcal{U} \colon$ the condition ($^{\bigstar} _{\bigstar}$) is satisfied for all connected components $C_1,C_2,C_3$ of $A\}$.
\end{center}

\begin{lemma}
\label{lemma 5.1}
Let $A\in \mathcal{U}$. Then there is $B\in \mathbf{R}(A)\cap \mathcal{U}^{\bigstar}_{\bigstar}$ such that $\mathbf{V}(A)=\mathbf{V}(B)$.
\end{lemma}
\begin{proof}
If $A\in \mathcal{U}^{\bigstar}_{\bigstar}$ then $B= A$.
Assume that $A$ does not fulfil $(^{\bigstar} _{\bigstar})$. 
For each connected component $C$ of $A$ let $U(C)$ be the set of all connected components of $A$ isomorphic to $C$. 
Denote
$$
\{U_{i}\colon i\in I\}=\{U(C)\colon C \ \ \textrm{is a connected component of}\ \ A, \ \ \card U(C)>2\}.
$$
Since $(^{\bigstar} _{\bigstar})$ is not valid in $A$, the set $I$ is nonempty. 
For each $i\in I$ take two distinct fixed elements of $U_{i}$ and denote them by 
$C_i$ and $C'_i$. 
Now put
$$
B=(A\setminus \bigcup\limits_{i\in I} \bigcup\limits_{K\in U_{i}} K)
\cup \bigcup\limits_{i\in I}(C_i\cup C'_i).
$$
Then $B\in \mathcal{U}^{\bigstar}_{\bigstar}$. 
If $i\in I$ and $C\in U_{i}$ then there exist isomorphisms
$$
\psi_{C} \colon C\to C_i,
$$
$$
\psi '_{C} \colon C\to C'_i.
$$
Define a mapping $\varphi\colon A\to B$ as follows:

if $x\in B$, then $\varphi(x)=x$, or

if $x\in C\in U_{i}\setminus \{C_i,C'_i\}$ for some $i\in I$, then $\varphi(x)=\psi_{C}(x)$.
\newline
This $\varphi$ is a retraction endomorphism on $A$, therefore 
$B\in \mathbf{R}(A)$ and $\mathbf{V}(B)\subset \mathbf{V}(A)$ according to Lemma~\ref{pom7}.

We need to see $A\in \mathbf{V}(B)$.
Let
$$
\kappa = \sup\{\card U_{i}\colon i\in I\}
$$
and let $J$ be a set with $\card J=\kappa$. Put $B_{j}=B$ for each $j\in J$ and let $D=\prod\limits_{j\in J}B_{j}$.
We are going to show that $A\in \mathbf{R}(D)$. If $a\in B$, then denote $\nu(a)=\overline{a}\in D$, where $\overline{a}(j)=a$ for each $j\in J$.
If $i\in I$, then the set
$$
Q_{i}=\{Y \colon Y \ \ \textrm{is a connected component of}\ \ D,\ \ Y(j)\in\{C_i, C'_i\}\ \ \textrm{for each} \ \ j\in J\}
$$
has at least $2^{\kappa}$ elements, thus there is an injection $\xi_{i}\colon U_{i} \to Q_{i}$. 
If $a\in C\in U_{i}-\{C_i, C'_i\}$, then define $\nu(a)=z$, where
$$
z(j)=
\begin{cases}
\psi_{C}(a) &  \text{if}\ \ j\in J \ \text{and}\ (\xi_{i}(C))(j)=C_i,\\
\psi '_{C}(a) &  \text{if}\ \ j\in J \ \text{and}\ (\xi_{i}(C))(j)=C'_i.
\end{cases}
$$
Denote $T=\nu(A)$. It can be verified that
\begin{enumerate}
\item $\nu \colon A\to T$ is an isomorphism,
\item $T$ is a retract of $D$.
\end{enumerate}
Thus
$$
A\in \mathbf{R}(D)\subseteq \mathbf{RP}(B)=\mathbf{V}(B).
$$
\end{proof}

\begin{lemma}
\label{lemma 5.2}
Let $A\in \mathcal{U}$.
Then there exist a set $K$ and algebras 
$B_{\xi}\in \mathcal{U}$ for each $ \xi \in K$
such that
\begin{enumerate}
\item
if $\xi \in K$, then every connected component of $B_{\xi}$ belongs to $\mathcal{S}$,
\item
$\mathbf{V}(A)=\mathbf{V}(\{B_{\xi}\colon \xi\in K\})$.
\end{enumerate}
\end{lemma}
\begin{proof}
Let $\{A_j, j\in J\}$ be a partition of $A$ into connected components.

By using Lemma~\ref{lemma 4.7} we obtain that for each $j\in J$ there exist a set $I_{j}$ and  $D_{ji}\in \mathbf{R}(A_{j})\cap \mathcal{S},i\in I_{j}$ such that 
 $A_{j}\in \mathbf{R}(\prod\limits_{i\in I_{j}}D_{ji})$.
Put $K$ the set of all mappings $\xi$ of $J$ into $\bigcup\limits_{j\in J}I_{j}$ such that
$\xi(j)\in I_{j}$ for each $j\in J$.
If $\xi \in K$, then put $B_{\xi}$ this algebra that 
 $\{D_{j\xi(j)},j\in J\}$ is a partition of $B_{\xi}$ into connected components.

If $\xi \in K$, then $D_{j\xi(j)}\in \mathbf{R}(A_{j})$ for every $j\in J$.
Thus $B_{\xi}\in \mathbf{R}(A)$ and 
$$\mathbf{V}(\{B_{\xi}\colon \xi\in K\})\subseteq \mathbf{V}(A)$$ according to Lemma~\ref{pom7}.

To see the opposite inclusion suppose that $\nu_{j}$ is an isomorphism of $A_{j}$ onto some retract $T_{j}$ of $\prod\limits_{i\in I_{j}}D_{ji}$.
Let us define a mapping $\nu$ of $A$ into $\prod\limits_{\xi \in K}B_{\xi}$ as follows. If $a\in A$, then there is a uniquely determined $j\in J$ with $a\in A_{j}$. Then $\nu_{j}(a)\in T_{j}\subseteq \prod\limits_{i\in I_{j}}D_{ji}$ and we put $\nu(a)=b$, where 
$b(\xi)=(\nu_{j}(a))(\xi(j))$ for each $\xi \in K$.

The mapping $\nu$ is injective, since if $a\in A_{j}$, $a'\in A_{m}$, $j,m\in J$, $\nu(a)=\nu(a' )$, then, for each $\xi\in K$,
$$
(\nu_{j}(a))(\xi(j))=(\nu_{m}(a' ))(\xi(m)),
$$ 
thus $j=m$, $\nu_{j}(a)= \nu_{j}(a' )$, hence $a=a' $. 
It can be shown that $\nu $ is a homomorphism, thus
$\nu(A)\cong A$.
Further, $\nu(A)$ is a retract of $\prod\limits_{\xi\in K}B_{\xi}$ since $A_{j}\in \mathbf{R}(\prod\limits_{i\in I_{j}}D_{ji})$.
Therefore we have 
 $A\in \mathbf{RP}(B_{\xi}\colon \xi\in K)= \mathbf{V}(B_{\xi}\colon \xi\in K)$.
\end{proof}

\begin{corollary}
\label{corollary 5.0.1}
Let $A$ be a monounary algebra. 
Then there are a set $I$ and monounary algebras 
$B_{i}\in \mathcal{U}^{\bigstar}_{\bigstar}$ for each $i\in I$ such that
\begin{enumerate}
\item[\rm(i)] $ \mathbf{V}(A)=\mathbf{V}(B_{i}\colon i\in I)$,
\item[\rm(ii)] if $i\in I$ and $C$ is a connected component of $B_{i}$, then 
$C\in \mathcal{S}$.
\end{enumerate}
\end{corollary}
\begin{proof}
The assertion is a consequence of Lemmas~\ref{lemma 5.2},~\ref{lemma 5.1} and~\ref{pom7}.
\end{proof}

\begin{corollary}
\label{cor 5.0.2}
Let $\mathcal{K}\subseteq \mathcal{U}$. 
Then there exists $\mathcal{L}\subseteq \mathcal{U}^{\bigstar}_{\bigstar}$ such that
\begin{enumerate}
\item[\rm(i)]
$\mathbf{V}(\mathcal{L})=\mathbf{V}(\mathcal{K})$,
\item[\rm(ii)] if $A\in \mathcal{L}$ and $C$ is a connected component of $A$, then 
$C\in \mathcal{S}$.
\end{enumerate} 
\end{corollary}

\begin{notation}
{\rm{
For $\beta\in\ord$ let 

$\mathcal{R}_{\beta}=\{ B\in \mathcal{M}_{\beta}\cap \mathcal{U}^{\bigstar}_{\bigstar}
\colon$ 
if $C$ is a connected component of $B$, then 
$C\in \mathcal{S}\}.$
}}
\end{notation}

\begin{lemma}
\label{lemma 5.3}
Let $\beta\in \ord$. 
Then $\mathcal{R}_{\beta}/\cong$ is a set such that 
\begin{align*}
\card(\mathcal{R}_{\beta}/\cong)\leqq 2^{(\tau(\beta))^{\aleph_{0}}}
\end{align*}
\end{lemma}
\begin{proof}
If $A\in \mathcal{R}_{\beta}$, then every connected component of $A$ is from $\mathcal{S}_{\beta}$.
In view of $A\in \mathcal{U}^{\bigstar}_{\bigstar}$ and Corollary~\ref{corollary 4.0.1} the algebra $A$ consists of at most
\begin{align*}
2\cdot \card(\mathcal{S}_{\beta}^{(1)}\cup \mathcal{S}_{\beta}^{(2)}) = (\tau(\beta))^{\aleph_{0}}
\end{align*}
connected components. 
For each connected component there are at most $(\tau(\beta))^{\aleph_{0}}$ possibilities, therefore
\begin{align*}
\card(\mathcal{R}_{\beta}/\cong)\leqq [(\tau(\beta))^{\aleph_{0}}]^{(\tau(\beta))^{\aleph_{0}}} = 2^{(\tau(\beta))^{\aleph_{0}}}.
\end{align*}
\end{proof}

\begin{proposition}
\label{proposition 5.1}
Let $\beta\in\ord$ and $\mathcal{V}\subseteq \mathcal{M}_{\beta}$ be a retract variety.
 Then $\mathcal{V}$ is set-principal.
\end{proposition}
\begin{proof}
If $A\in\mathcal{V}$, then Corollary~\ref{corollary 5.0.1} implies that there are a set $I_{A}$ and monounary algebras $B_{i}^{A}\in \mathcal{U}^{\bigstar}_{\bigstar}$ for each $i\in I_{A}$ such that
\begin{enumerate}
\item[\rm(i)] $ \mathbf{V}(A)=\mathbf{V}(B_{i}^{A}\colon i\in I_{A})$,
\item[\rm(ii)]  if $i\in I_{A}$ and $C$ is a connected component of $B_{i}^{A}$, 
then $C\in \mathcal{S}$.
\end{enumerate}
Then the assumption $\mathcal{V}\subseteq \mathcal{M}_{\beta}$ implies
$B_{i}^{A}\in \mathcal{R}_{\beta}$ for each $i\in I_{A}$.
We get
\begin{align*}
\mathcal{V} = \mathbf{V}(\{A \colon A\in\mathcal{V}\}) = 
\mathbf{V}(\{B_{i}^{A} \colon A\in\mathcal{V}, i\in I_{A}\}) \subseteq \mathbf{V}(\mathcal{R}_{\beta}).
\end{align*}
and Lemma~\ref{lemma 5.3} implies that the retract variety $\mathcal{V}$ is set-principal.
\end{proof}

\begin{theorem}\label{(th2)}
Let $\mathcal{V}$ be a retract variety of monounary algebras. The following conditions are equivalent:
\begin{enumerate}
\item $\mathcal{V}$ is set-principal,
\item there is $\beta\in \ord$ such that  if $A\in \mathcal{V}$, $x\in A\setminus A^{(\infty)}$, then $\s_{f}(x)\leqq\beta$.
\end{enumerate}
\end{theorem}
\begin{proof}
It follows from Propositions~\ref{proposition 5.1} and~\ref{prop_SP}. 
\end{proof}

\begin{proposition}
\label{prop 5.3}
There exists a retract variety $\mathcal{T}$ of monounary algebras such that $\mathcal{T}$ is not set-principal.
\end{proposition}
\begin{proof}
For $\alpha\in\ord$ there exists a (connected) monounary algebra $A_{\alpha}$ and an element $a_{\alpha}\in A_{\alpha}$ such that $\s_{f}(a_{\alpha})=\alpha$. Then Theorem~\ref{(th2)} implies that the retract variety $\mathbf{V}(\{A_{\alpha} \colon \alpha \in \ord \})$ is not set-principal.
\end{proof}

\end{document}